\documentclass{amsart}
\usepackage{amsmath}
\usepackage{color}
\usepackage{mathtools}
\usepackage[colorlinks,
            linkcolor=blue,
            anchorcolor=blue,
            citecolor=red,
            bookmarksnumbered]{hyperref}

\newcommand{\bM}{\bar{M}}

\newcommand{\bfR}{\hbox{\bbbld R}}

\newtheorem{thm}{Theorem}
\newtheorem{lem}[thm]{Lemma}
\newtheorem{remark}{Remark}

\font\bbbld=msbm10 scaled\magstephalf

	\newcommand{\ul}{\underline}

\begin{document}
	
	\title{On estimates for augmented Hessian type parabolic
		equations on Riemannian manifolds}
	
\author{Yang Jiao}
\address{School of Mathematics, Harbin Institute of Technology,
	Harbin, Heilongjiang 150001, China}
\email{18b912015@stu.hit.edu.cn}
	
	\begin{abstract}
		The author extends
		previous results to general classes of equations under weaker assumptions obtained in
		2016 by Bao, Dong and Jiao concerning the study of the regularity of solutions for the
		first initial-boundary value problem for parabolic Hessian equations on Riemannian manifolds.
		
	\end{abstract}
	
	
	\maketitle
	
	\section{Introduction}
	
	Let $(M^n, g)$ be
	a compact Riemannian manifold of
	dimension $n \geq 2$ with smooth boundary $\partial M$ and $\bM:= M \cup \partial M$.
	Define
	$M_T = M \times (0,T] \subset M \times \mathbb{R}$, $\mathcal{P} M_T
	= B M_T \cup S M_T$ is the parabolic boundary of $M_T$ with $B M_T = M
	\times \{0\}$ and $S M_T = \partial M \times [0,T]$.
	In \cite{BaoDongJiao.2016}, the authors derived $C^{2}$ estimates for solutions of
	the first initial-boundary value problem of parabolic Hessian equations in the form
	\begin{equation}
		\label{BDJeqn}
		f (\lambda(\nabla^{2}u + \chi (x, t)),-u_{t})= \psi(x, t) ,
	\end{equation}
	where $f$ is a symmetric smooth function of $n+1$ variables.

	In this paper, we apply an exponential barrier from \cite{JiangTrudinger.2020} where Jiang-Trudinger treat the corresponding elliptic problems in $\mathbb{R}^{n}$ to study \eqref{BDJeqn} in the general augmented Hessian form
	\begin{equation}
		\label{eqn}
		f (\lambda(\nabla^{2}u + A (x, t, \nabla u)), - u_t)  = \psi (x, t, \nabla u)
	\end{equation}
	in $M_T$ with boundary condition
	\begin{equation}
		\label{eqn-bd}
		u = \varphi \mbox{ on }\mathcal{P} M_T,
	\end{equation}
	where $\nabla^{2}u + A (x, t, \nabla u)$ is called augmented Hessian,
	$\nabla u$ and $\nabla^2 u$ denote the gradient and the
	Hessian of $u (x, t)$ with respect to $x \in M$ respectively, $u_t = D_t u$ is the derivative
	of $u (x, t)$ with respect to $t \in [0, T]$,
	$A [u] = A (x, t, \nabla u)$ is a $(0,2)$ tensor on
	$\overline{M}$ which may depend on $t \in [0, T]$ and $\nabla u$,
	and
	\[ \lambda (\nabla^2 u + A [u]) = (\lambda_1 ,\ldots,\lambda_n) \]
	denotes the eigenvalues of $\nabla^2 u + A [u]$ with respect to the metric $g$.
	
	As in \cite{Jiao.2015}, throughout the paper we assume $A [u]$ is smooth on $\overline{M_T}$ for $u \in C^{\infty} (\overline{M_T})$,
	$\psi \in C^{\infty} (T^*\overline{M} \times [0,T])$. We shall write
	$\psi = \psi (x, t, p)$ for $(x, p) \in T^*\overline{M}$ and $t \in [0, T]$.
	Note that for fixed $(x, t) \in \overline{M_T}$ and $p \in T^*_x M$,
	\[
	A (x, t, p): T^*_x M \times T^*_x M \rightarrow \bfR
	\]
	is a symmetric bilinear map. We shall use the notation
	\[
	A^{\xi \eta} (x, t, \cdot) := A (x, t, \cdot) (\xi, \eta), \;\;
	\xi,  \eta \in T^*_x M.
	\]
	For a function $v \in C^{2} (M_T)$, we write $A [v] := A (x, t, \nabla v)$,
	$A^{\xi \eta} [v] := A^{\xi \eta} (x, t, \nabla v)$ and $\psi[u]:=\psi(x, t, \nabla u)$.
	
	There are many different $A$ in conformal geometry, the optimal transportation satisfies, the isometric embedding, reflector design and other research fields, we recommend readers see subsection 3.8 in \cite{TrudingerWang2008} and references therein for the Monge-Amp\`ere type equations arising in applications.
	
	We are concerned in this work with the \emph{a priori} estimates of admissible solutions
	to \eqref{eqn} with boundary condition.
	The use of the exponential barrier allows us to relax the concavity assumption of $A$ to Ma-Trudinger-Wang conditions(see \cite{MaTrudingerWang.2005}).  By the perturbation method of subsolutions
	in \cite{JiangTrudinger.2020} (see Remark 2.2 in \cite{JiangTrudingerYang.2014} for details), we can obtain strict subsolutions from non-strict subsulutions which simplifies
	the proofs and relaxes some restrictions to $f$ in the estimates of $|u_{t}|$.
	
	Our treatment here will also work for parabolic equations in the form
	\begin{equation}\label{eqn'}
		f (\lambda(\nabla^{2}u + A (x, t, \nabla u))) - u_t  = \psi (x, t, \nabla u)
	\end{equation}
	with slight modification.
	Note that we do not require \emph{a priori} bound
	of $|u_{t}|$ in the study of \eqref{eqn'}.

	The idea of this paper is mainly from Guan-Jiao \cite{GuanJiao.2016}
	and Jiang-Trudinger \cite{JiangTrudinger.2020} where those authors studied
	the second order estimates for the elliptic counterpart of \eqref{eqn}:
	\begin{equation}
		\label{3I-10}
		f (\lambda (\nabla^2 u + A (x, u, \nabla u)))
		= \psi (x, u, \nabla u).
	\end{equation}

	The first initial-boundary value problem for equation of form \eqref{eqn'} in $\mathbb{R}^n$
	with $A \equiv 0$ and $\psi = \psi (x,t)$ was studied by Ivochkina-Ladyzhenskaya in \cite{IvochkinaLadyzhenskaya.1994}
	(when $f = \sigma_n^{1/n}$) and \cite{IvochkinaLadyzhenskaya.1995}. In recent years, Jiao-Sui \cite{JiaoSui.2014} treated the case that
	$A \equiv \chi (x,t)$ and $\psi = \psi (x,t)$ on Riemannian manifolds and
	Jiao \cite{Jiao.2015} extend their results to the form
	$$
	f (\lambda(\nabla^{2}u + A (x, t, \nabla u))) - u_t  = \psi (x, t, u, \nabla u)
	$$
	by the method using in the corresponding elliptic problems.
	
	Krylov in \cite{Krylov.1976} treated \eqref{eqn} in the parabolic Monge–Amp\`{e}re form
	$$
	-u_{t}\det(\nabla^{2}u+A)=\psi^{n+1}
	$$
	in $\mathbb{R}^n$, where $A \equiv 0$ and $\psi = \psi (x,t)$.
	In \cite{Lieberman.1996}, Lieberman studied the first initial–boundary value problem of \eqref{eqn} when $A = 0$ and $\psi$ may
	depend on $u$ and $\nabla u$ in a bounded domain under various conditions.

	For the elliptic Hessian equations, we refer
	the readers to Li \cite{Li.1990}, Urbas \cite[367--377]{Urbas.2002}, Guan \cite{Guan.1999,Guan.2014}, Guan-Jiao \cite{GuanJiao.2015} , Jiang-Trudinger \cite{JiangTrudinger.2020} and
	their references.

	Following \cite{CaffarelliNirenbergSpruck.1985}, in which the authors studied the corresponding elliptic equations
	in $\mathbb{R}^n$, $f \in C^\infty (\Gamma) \cap C^0 (\overline{\Gamma})$
	is assumed to be defined on $\Gamma$, where $\Gamma$ is an open, convex,
	symmetric proper subcone of $\mathbb{R}^{n+1}$ with vertex at the origin and
	\[\Gamma^{+} \equiv \{\lambda \in \mathbb{R}^{n+1}:\mbox{ each component }
	\lambda_{i} > 0\} \subseteq \Gamma,\]
	and to satisfy the following structure conditions in this paper:
	\begin{equation}
		\label{f1}
		f_{i} \equiv \frac{\partial f}{\partial \lambda_{i}} > 0 \mbox{ in } \Gamma,\ \ 1\leq i\leq n+1,
	\end{equation}
	\begin{equation}
		\label{f2}
		f\mbox{ is concave in }\Gamma,
	\end{equation}
	and
	\begin{equation}
		\label{f5}
		\delta_{\psi , f}\equiv\inf_{M_T} \psi -\sup_{\partial \Gamma} f>0 ,
		\ \  \mbox{where} \; \sup_{\partial \Gamma} f\equiv
		\sup_{\lambda_{0}\in \partial \Gamma} \limsup_{\lambda\rightarrow\lambda_{0}} f(\lambda) .
	\end{equation}

	Typical examples are $f = \sigma^{1/k}_k$ and $f = (\sigma_k / \sigma_l)^{1/(k - l)}$, $1 \leq l < k \leq n$,
	defined in the cone
	$$
	\Gamma_{k} = \{\lambda \in \mathbb{R}^{n}: \sigma_{j} (\lambda) > 0, j = 1, \ldots, k\}
	$$
	and
	$f = (\mathcal{M}_k)^{1/\binom{n}{k}}$ defined in
	$$
	M_k  = \{\lambda \in \bfR^n:
	\lambda_{i_1} + \cdots + \lambda_{i_k} > 0\},
	$$
	where $\sigma_{k} (\lambda)$ are the $k$th elementary symmetric functions and
	$\mathcal{M}_k$ are the $p$-plurisubharmonic functions defined by
	$$
	\sigma_{k} (\lambda) = \sum_ {i_{1} < \ldots < i_{k}}
	\lambda_{i_{1}}\cdots \lambda_{i_{k}}, \;\; 1 \leq k \leq n
	$$
	and
	$$
	\mathcal{M}_k (\lambda) = \prod_{i_1 < \cdots < i_k}
	(\lambda_{i_1} + \cdots + \lambda_{i_k}), \;\; 1 \leq k \leq n
	$$
	respectively.
	When $k=n$, $f=\sigma_{n}^{\frac{1}{n}}$ is the famous Monge-Amp\`ere equation arising in many research fields such as
	conformal geometry,  optimal transportation, isometric embedding and reflector designs, see the survey \cite{TrudingerWang2008} and references therein.
	
	We define a function $u(x,t)$ to be admissible if $(\lambda(\nabla^{2} u + A[u]), -u_t) \in \Gamma$ in $M \times [0, T]$.
	It is shown in \cite{CaffarelliNirenbergSpruck.1985} that \eqref{f1} ensures that Eq \eqref{eqn}
	is parabolic for admissible solutions. \eqref{f2} means that the function $F$ defined
	by $F (A, \tau) = f (\lambda [A], \tau)$ is concave for $(A, \tau)$ with
	$(\lambda [A], \tau) \in \Gamma$, where $A$ is in the set of $n \times n$
	symmetric matrices $\mathcal{S}^{n \times n}$. Moreover, when $\{U_{ij}\}$ is diagonal so is $\{F^{ij}\}$, and the
	following identities hold
	\[   F^{ij} U_{ij} = \sum f_i \lambda_i, \;\; F^{ij} U_{ik} U_{kj} = \sum f_i \lambda_i^2,
	\;\; \lambda (U) = (\lambda_1, \ldots, \lambda_n). \]
	
	We define a function $\overline{u}$ to be a admissible viscosity supersolution of
	\eqref{eqn}
	if
	$$
	f (\lambda(\nabla^{2}\phi(\hat{x}, \hat{t}) + A (\hat{x}, \hat{t}, \nabla \phi(\hat{x}, \hat{t})), - \phi_t(\hat{x}, \hat{t}))
	\leq \psi (\hat{x}, \hat{t}, \nabla \phi(\hat{x}, \hat{t}))
	$$
	whenever $\phi\in C^{2} (M_T)$ is a admissible function and $(\hat{x}, \hat{t})\in M_T$ is a local minimum of $\overline{u}-\phi$.

	In this paper
	we assume that there exists an admissible function
	$\underline{u} \in C^{2} (\bM_T)$ satisfying
	\begin{equation}
		\label{sub}
		\left\{ \begin{aligned}
			f(\lambda(\nabla^{2} \underline{u} + A [\ul u]),  - \underline{u}_{t})
			& \geq \psi(x, t, \nabla \ul u) &&\mbox{ in } M \times [0, T], \\
			& \ul u = \varphi               && \mbox{ on } \partial M \times [0, T],\\
			& \ul u \leq \varphi            && \mbox{ on } M \times \{0\}.
		\end{aligned} \right.
	\end{equation}
	
	A $(0,2)$ tensor $B$ is called regular (strictly regular), if
	$$
	\sum_{i,j,k,l}^{n}B^{ij}_{p_{k},p_{l}}(x, t, p)\xi_{i}\xi_{j}\eta_{k}\eta_{l}\geq 0(>0)
	$$
	for all $(x,t,p)\in M\times [0,T] \times \bfR^{n}$, $\xi,\eta\in T^{*}_{x}M$ and $g(\xi, \eta)=0$.
	
	The regular condition, well known as MTW condition, was first introduced by Ma, Trudinger and Wang in \cite{MaTrudingerWang.2005} for the study of optimal transportation in its strict form, and used in
	\cite{Andriyanova.2019}, \cite{JiangTrudinger.2020} and other relevant problems. It is natural to consider MTW conditions instead of normal concavity assumptions on $A$.
	Examples in \cite{MaTrudingerWang.2005} shows that there exists a tensor $A$, without convexity respect to $p$,
	derived from special cost functions satisfying this regular condition. There are many results about MTW conditions, see, for instance, \cite{DuLi.2014,FigalliRiffordVillani.2010,FigalliRiffordVillani.2011,GoodrichScapellato.2022, LoeperTrudinger.2021,Ragusa.2001}
	and references therein.
	
	We now begin to formulate the main theorems of this paper.
	\begin{thm}
		\label{jjy-th1}
		Let $u \in C^4 (\bM_T)$ be an
		admissible solution of (\ref{eqn}).
		Suppose \eqref{f1}--\eqref{f5} and \eqref{sub} hold.
		Assume, in addition, that
		\begin{equation}
			\label{A2}
			\mbox{$\psi (x,t,p)$ is convex
				in $p$},
		\end{equation}
		\begin{equation}\label{A3w}
			\mbox{$-A^{\xi \xi} (x,t,p)$ is regular},
		\end{equation}
		then
		\begin{equation}
			\label{hess-a10}
			\max_{\bM_T} |\nabla^2 u| \leq
			C_1 \big(1 + \max_{\mathcal{P} M_T}|\nabla^2 u|\big),
		\end{equation}
		where $C_1 > 0$ depends on $|u|_{C^1 (\bM_T)}$, $|u_{t}|_{C^0 (\bM_T)}$ and $|\ul u|_{C^2 (\bM_T)}$.
		Suppose that $u$ also satisfies the boundary condition
		\eqref{eqn-bd} and, in addition,
		assume that there exists a function $\Theta \in C^2(BM_T)$
		such that $\Theta=-\varphi_t$ on $\partial M\times \{0\}$ and
		\begin{equation}
			\label{comp1}
			(\lambda (\nabla^2 \varphi (x, 0) + A [\varphi (x, 0)]), \Theta(x))\in \Gamma,\ \ \forall x \in \bM,
		\end{equation}
		and that
		\begin{equation}
			\label{comp2}
			f (\lambda (\nabla^2 \varphi (x, 0) + A [\varphi (x, 0)]), - \varphi_t (x, 0))
			= \psi [\varphi (x, 0)],\ \ \forall x \in \partial M,
		\end{equation}
		for each $(x, t) \in SM_T$ and $p \in T^*_x \bM$ .
		Then there exists $C_2 > 0$ depending on
		$|u|_{C^1 (\bM_T)}$, $|u_{t}|_{C^0 (\bM_T)}$, $|\ul u|_{C^2 (\bM_T)}$ and
		$|\varphi|_{C^4 (\mathcal{P} M_T)}$ such that
		\begin{equation}
			\label{hess-a10b}
			\max_{\mathcal{P} M_T}|\nabla^2 u| \leq C_2.
		\end{equation}
	\end{thm}

	Combining with the gradient estimates and the estimates of $|u_{t}|$, we can prove
	the following theorem immediately.
	
	\begin{thm}
		\label{thm-main}
		Let $u \in C^4 (\bM_T)$ be an
		admissible solution of (\ref{eqn}) in $M_T$ with
		$u \geq \ul u$ in $M_T$ and $u = \varphi$ on $\mathcal{P} M_T$.
		Suppose \eqref{f1}--\eqref{A3w} and \eqref{comp1}--\eqref{comp2} hold.
		Assume, in addition, for every $C>0$, there is a constant $R=R(C)$ such that
		\begin{equation}\label{ut-condition}
			f(R\textbf{1})>C,
		\end{equation}
		where $\textbf{1}=(1, \ldots, 1)\in \mathbb{R}^{n+1}$.
		Assume also there exist a bounded admissible viscosity supersolution $\overline{u}$ of \eqref{eqn}
		satisfying $\overline{u}\geq\varphi$ on $\mathcal{P} M_T$.
		Then we have
		\begin{equation}
			\label{gsui-3}
			|u|_{C^{2} (\bM_T)}\leq C,
		\end{equation}
		where $C > 0$ depends on $n$, $M$ and $|\ul u|_{C^2(\bM_T)}$
		under the additional assumptions \eqref{G1}--\eqref{Grd1} in Section 3.
	\end{thm}
	
	The assumptions of the existence of bounded viscosity supersolution and the additional conditions \eqref{G1}--\eqref{Grd1} are only used to derive $C^0$ and $C^1$ estimates.
	\eqref{ut-condition} is used in the estimates of $|u_{t}|$ and can be dropped if $\ul u$ is strict subsolution.
	Both \eqref{ut-condition} and \eqref{Grd1} hold for many operators such as the famous Monge-Amp\` ere operator or more general k-Hessian operator $\sigma_{k}^{1/k}$.

	The outline of this paper is as follows. In Section 2, we present some preliminaries and give a proof of Lemma~\ref{lemma2.1}.
	The solution bound and the gradient bound are derived in Section 3 while an \emph{a priori}
	estimates for $u_{t}$ is obtained in Section 4.  Finally we
	establish the global and boundary $C^{2}$ estimates in Section 5 and Section 6 respectively.

	\section{Preliminaries}
	\label{gj-P}
	\setcounter{equation}{0}
	\medskip
	
	Throughout the paper $\nabla$ denotes the Levi-Civita connection
	of $(M^n, g)$.
	
	Let $u \in C^4 (\bM_T)$ be an admissible solution of
	Eq~\eqref{eqn}.
	For simplicity we shall denote $U := \nabla^2 u + A (x, t, \nabla u)$
	and $\ul U := \nabla^2 \ul u + A (x, t, \nabla \ul u)$.
	Moreover, we denote,
	\[F^{ij} = \frac{\partial F}{\partial h_{ij}} (U, -u_{t}), \;\;
	F^{\tau}=\frac{\partial F}{\partial \tau}(U, -u_{t}), \]
	\[ F^{ij, kl} = \frac{\partial^2 F}{\partial h_{ij} \partial h_{kl}} (U, -u_{t}), \;\;
	F^{ij, \tau} = \frac{\partial^2 F}{\partial h_{ij} \partial \tau} (U, -u_{t}), \;\;
	F^{\tau, \tau} = \frac{\partial^2 F}{\partial^{2} \tau } (U, -u_{t})\]
	and, under a local frame $e_1, \ldots, e_n$,
	\[ U_{ij} \equiv U (e_i, e_j) = \nabla_{ij} u + A^{ij} (x, t, \nabla u), \]
	$$
	\begin{aligned}
		\nabla_k U_{ij}
		\equiv \,& \nabla U (e_i, e_j, e_k)
		= \nabla_{kij} u + \nabla_k A^{ij} (x, t, \nabla u)  \\
		\equiv \,& \nabla_{kij} u  + A_{k}^{ij} (x, t, \nabla u)
		+ A^{ij}_{p_l} (x, t, \nabla u) \nabla_{kl} u,
	\end{aligned}
	$$
	$$
	\begin{aligned}
		(U_{ij})_t
		\equiv \,& (U (e_i, e_j))_t
		= (\nabla_{ij} u)_t + A^{ij}_t (x, t, \nabla u)
		+ A^{ij}_{p_l} (x, t, \nabla u) (\nabla_{l} u)_t \\
		\equiv \,& \nabla_{ij} u_t  + A_{t}^{ij} (x, t, \nabla u)
		+ A^{ij}_{p_l} (x, t, \nabla u) \nabla_{l} u_t,
	\end{aligned}
	$$
where $A^{ij} = A^{e_{i} e_{j}}$ and
$A_{k}^{ij}$ denotes the {\em partial} covariant derivative of $A$
when viewed as depending on $x \in M$ only, while the meanings of
$A^{ij}_t$ and $A^{ij}_{p_l}$, etc are obvious. Similarly we can calculate
$\nabla_{kl} U_{ij} = \nabla_k \nabla_l U_{ij} - \Gamma_{kl}^m \nabla_m U_{ij}$, etc.

It is convenient to express the regular condition of $-A$ in the equivalent form as in \cite{JiangTrudinger.2017},
\begin{equation}\label{A3w-1}
	-A^{ij}_{p_{k}p_{l}}\xi_{i}\xi_{j}\eta_{k}\eta_{l}\geq -2\overline{\lambda}|\xi||\eta|g(\xi\cdot\eta),
\end{equation}
for all $\xi, \eta\in \mathbb{R}^{n}$, where $\overline{\lambda}$ is a non-negative function in $C^{0}(\overline{M_{T}}\times \mathbb{R}^{n})$,
depending on $\nabla_{p}A$. Hence, we have, for any non-negative symmetric matrix $F^{ij}$ and $\epsilon \in (0,1]$,
\begin{equation}\label{A3w-1-1}
	-F^{ij}A^{ij}_{p_{k}p_{l}}\eta_{k}\eta_{l}\geq -\overline{\lambda}\big(\epsilon\sum F^{ii}|\eta|^{2}+\frac{1}{\epsilon}F^{ij}\eta_{i}\eta_{j}\big).
\end{equation}

Define the linear operator $\mathcal{L}$ locally by
\[\mathcal{L} v = F^{ij} \nabla_{ij} v + (F^{ij} A^{ij}_{p_k} - \psi_{p_k}) \nabla_k v - F^{\tau}v_t\]
for $v \in C^{2} (M_T)$.

A crucial lemma was proved by Jiang-Trudinger for elliptic type equations in Lemma 2.1(ii)
in \cite{JiangTrudinger.2020} for $M=\mathbb{R}^{n}$, we extend their results to the parabolic case.
Note that their perturbation of non-strict subsolution, which make a non-strict subsolution to be strict, only holds near the boundary in the Riemannian manifolds case. Therefore we shall apply a classification technique from \cite{GuanJiao.2016} to deal with global estimates.

Let $\mu(x,t)=\lambda(\nabla^{2}\ul u(x,t)+A[\ul u])$ and note that $\{\mu(x,t):(x,t)\in M_{T}\}$ is a compact subset of positive cone $\Gamma^{+}$ since \eqref{f1}. There exists uniform constant $\beta\in(0,\frac{1}{2\sqrt{n}})$ such that
\begin{equation}\label{2.0}
	\nu_{\mu}-2\beta \textbf{1}\in \Gamma^{+},\;\forall x \in \bM_{T},
\end{equation}
where $\nu_{\lambda}:=Df(\lambda)/|Df(\lambda)|$  is the unit normal vector to the level hypersurface $\partial\Gamma^{f(\lambda)}$ for $\lambda \in \Gamma$ and $\textbf{1}=(1, \ldots, 1)\in \mathbb{R}^{n+1}$.

For fixed $(x_{0},t_{0})$, we consider two cases: (i) $|\nu_{\mu}-\nu_{\lambda}|\geq \beta$ and (ii) $|\nu_{\mu}-\nu_{\lambda}|< \beta$. In case (i), we shall modify Jiang-Trduinger's Lemma~2.1 \cite{JiangTrudinger.2020}.
First, we need the following lemma, its proof can be found in Lemma~2.2 \cite{GuanShiSui2015}.
\begin{lem}
	\label{GSSlem}
	Let $K$ be a compact subset of $\Gamma$ and $\beta>0$. There is a constant $\epsilon>0$ such that, for any
	$\mu\in K$ and $\lambda\in\Gamma$ with $|\nu_{\mu}-\nu_{\lambda}|\geq \beta$,
	\begin{equation}\label{GSS}
		\sum f_{i}(\mu_{i}-\lambda_{i})\geq f(\mu)-f(\lambda)+\epsilon\Big(1+\sum f_{i}(\lambda)\Big).
	\end{equation}
\end{lem}
It follows from Lemma~6.2 in \cite{CaffarelliNirenbergSpruck.1985} and Lemma~\eqref{GSS} that
\begin{equation}\label{case 1_0}
	F^{ij}({\ul U}_{ij}-U_{ij})\geq F(\ul U, -{\ul u}_{t})-F(U,-u_{t})+\epsilon(1+\sum F^{ii}+F^{\tau}).
\end{equation}
We now prove the crucial lemma for case (i).
\begin{lem}
	\label{lemma2.1}
	Let $u\in C^{2}(\bM_{T})$ be an admissible solution of Eq \eqref{eqn}
	Suppose $|\nu_{\mu}-\nu_{\lambda}|\geq \beta$.
	Assume $F$ satisfies \eqref{f1}--\eqref{f2} and \eqref{sub}--\eqref{A3w} hold. Then there exist positive constants $K$ and $\epsilon$ , depending on $M_{T}$, $A$,
	$|u|_{C^{1}(\bM_{T})}$ and $|\ul u|_{C^{1}(\bM_{T})}$
	such that
	\begin{equation}\label{lem2.1}
		\mathcal{L}\eta>\epsilon(1+\sum F^{ii}+F^{\tau}),
	\end{equation}
	where $\eta=e^{K(\ul u-u)}$.
\end{lem}

\begin{proof}
	By \eqref{case 1_0}, we have
	\begin{equation}\label{Lem2.1-1}
		\begin{aligned}
			\mathcal{L}(\ul u-u)=\,&F^{ij}\{[\ul U_{ij}-U_{ij}]-F^{\tau}[\ul u_{t}-u_{t}]+A^{ij}_{p_{k}}D_{k}(\ul u-u)\\
			\,&     -A^{ij}(x,t,D\ul u)+A^{ij}(x,t,Du)\}-\psi_{p_{k}}\nabla_{k}(\ul u-u)\\
			\geq\,& F(\ul U,-\ul u_{t})-F(U,-u_{t})-\psi_{p_{k}}\nabla_{k}(\ul u-u)\\
			\,&     -\frac{1}{2}F^{ij}A^{ij}_{p_{k},p_{l}}(x,t,\hat{p})D_{k}(\ul u-u)D_{l}(\ul u-u)\\
			\,&     +\epsilon(1+\sum F^{ii}+F^{\tau})\\
			\geq\,& -\frac{1}{2}F^{ij}A^{ij}_{p_{k},p_{l}}(x,t,\hat{p})D_{k}(\ul u-u)D_{l}(\ul u-u)\\
			\,&     +\epsilon(1+\sum F^{ii}+F^{\tau})
		\end{aligned}
	\end{equation}
	by Taylor's formula and the convexity of $\psi$, where $\hat{p}=\theta\nabla u+(1-\theta)\nabla\ul u$ for some $\theta\in (0,1)$. Thus
	\begin{equation}\label{Lem2.1-2}
		\begin{aligned}
			\mathcal{L}e^{K(\ul u-u)}=\,&Ke^{K(\ul u-u)}[\mathcal{L}(\ul u-u)+KF^{ij}D_{i}(\ul u-u)D_{j}(\ul u-u)]\\
			\geq\,& Ke^{K(\ul u-u)}\Big\{-\frac{1}{2}F^{ij}A^{ij}_{p_{k},p_{l}}(x,t,\hat{p})D_{k}(\ul u-u)D_{l}(\ul u-u)\\
			\,&+KF^{ij}D_{i}(\ul u-u)D_{j}(\ul u-u)+\epsilon(1+\sum F^{ii}+F^{\tau})\Big\}.
		\end{aligned}
	\end{equation}
	Since $A$ is regular, by \eqref{A3w-1-1}, we obtain
	$$
	\begin{aligned}
		\epsilon\sum F^{ii}\,&-\frac{1}{2}F^{ij}A^{ij}_{p_{k},p_{l}}(x,t,\hat{p})D_{k}(\ul u-u)D_{l}(\ul u-u)
		+KF^{ij}D_{i}(\ul u-u)D_{j}(\ul u-u)\\
		\,&\geq \Big(\epsilon-\frac{\overline{\lambda}\epsilon_1}{2}|D(\ul u-u)|^{2}\Big)\sum F^{ii}
		+\Big(K-\frac{\overline{\lambda}}{2\epsilon_1}\Big)F^{ij}D_{i}(\ul u-u)D_{j}(\ul u-u)\\
		\,&\geq \frac{\epsilon}{2}\sum F^{ii}
	\end{aligned}
	$$
by successively fixing $\epsilon_{1}$ and $K$.

Therefore, by \eqref{Lem2.1-2}, we have
\begin{equation}\label{Lem2.1-4}
	\mathcal{L}e^{K(\ul u-u)}\geq Ke^{K(\ul u-u)}\Big(\frac{\epsilon}{2}(1+\sum F^{ii}+F^{\tau})\Big)\geq
	\epsilon_{0}(1+\sum F^{ii}+F^{\tau})
\end{equation}
for some positive constant $\epsilon_{0}$.
\end{proof}

Next, in case (ii), we have $\nu_{\lambda}-\beta \textbf{1}\in \Gamma^{+}$. Thus we derive
\begin{equation}\label{case 2}
F^{ii}\geq \frac{\beta}{\sqrt{n+1}}\sum F^{ii}\;\;\forall 1\leq i \leq n+1.
\end{equation}

\begin{remark}\label{rk1}
If $\ul u$ is a strict subsolution or $M=\mathbb{R}^{n}$, then we can derive \eqref{lem2.1} without the assumption $|\nu_{\mu}-\nu_{\lambda}|\geq \beta$.
Actually, when $M=\mathbb{R}^{n}$, let $d(x)=dist(x,\partial M)$, by consider $\ul u+ae^{bx_{1}}$ and $\ul u+a(e^{bd}-1)$ for interior and near boundary respectively in $\mathbb{R}^{n}$, a strict subsolution can be derived from a non-strict one, see remark 2.2 in \cite{JiangTrudingerYang.2014}.
Then \eqref{lem2.1} will be obtained by Jiang-Trudinger's proof with a little modification.
\end{remark}



\section{Gradient estimates}
In this section, we derive the gradient estimates.
We introduce the following growth conditions: When $|p|$ is sufficiently large,
\begin{equation}
\label{G1}
p \cdot \nabla_x \psi (x, t, p),
\; p \cdot \nabla_x A^{\xi \xi} (x, t, p)/|\xi|^2
\leq \bar{\psi}_{1}(x, t)  (1 + |p|^{\gamma}),
\end{equation}
\begin{equation}
\label{G2}
|p \cdot D_p \psi (x, t, p)|,
\; |p \cdot D_p A^{\xi \xi} (x, t, p)|/|\xi|^2
\leq \bar{\psi}_{2}(x, t)  (1 + |p|^{\gamma})
\end{equation}
and
\begin{equation}
\label{G3}
|A^{\xi \eta} (x, t, p)|
\leq \bar{\psi}_{3} (x, t) |\xi||\eta| (1 + |p|^{\gamma_{1}})
\;\; \forall \, \xi, \eta \in T^{*}_x \bM
\end{equation}
hold for some functions $\bar{\psi}_1, \bar{\psi}_2
,\bar{\psi}_3
\geq 0$,  and constants
$\gamma\in(0,4)$ and $\gamma_{1}\in(0,2)$.

By the existence of viscosity supersolution $\overline{u}$ and classical subsolution $\ul u$, we have
$$
\max_{\bM_T}|u|\leq C.
$$
Since $u$ is admissible, we have
\[
0 <   \triangle u + \mathrm{tr} A (x, t, \nabla u) -u_{t}.
\]
The boundary gradient estimates are derived by subsolution $\ul u$ for the lower bound and by \eqref{G3} with the method of Lemma 10.1 in \cite{Lieberman.1996}
for the upper bound.

\begin{thm}
\label{p-th0}
Let $u \in C^3 (\bM_T)$ be an
admissible solution of (\ref{eqn}).
Suppose \eqref{f1}--\eqref{f2} and \eqref{G1}--\eqref{G3} hold.
Assume, in addition, that
\begin{equation}
	\label{Grd1}
	f_{j}\geq\nu_{0}(1+\sum^{n+1}_{i=1}f_{i})\;\; \mbox{for any}\;  \lambda \in \Gamma\; \mbox{with} \;\lambda_{j}<0,
\end{equation}
where $\nu_{0}$ is a uniform positive constant.
Then
\begin{equation}
	\label{3I-R60}
	\max_{\bM_T} |\nabla u|
	\leq C_3 \big(1 + \max_{\mathcal{P} M_T} |\nabla u|\big),
\end{equation}
where $C_3$ is a positive constant depending on $|u|_{C^0 (\bM_T)}$ and other known data.
\end{thm}

\begin{proof}
Let $\phi\in C^{2}(\bM_{T})$ is a positive function to be determined. Suppose $|\nabla u|\phi^{-a}$ achieves a
positive maximum at an interior point $(x_{0},t_{0})\in \bM_{T}-\mathcal{P} M_T$ where $a<1$ is a constant. Choose a smooth orthonormal local frame $e_{1},\ldots,e_{n}$
about $(x_{0},t_{0})$ such that $\nabla_{e_{i}}e_{j}=0$ at $(x_{0},t_{0})$ if $i\neq j$ and $\{U_{ij}\}$ is diagonal.
Define $v=\log |\nabla u|-a\log \phi$, then
the function $v$ also attains its maximum at $(x_{0},t_{0})$ where, for $i=1, \ldots, n$,
\begin{equation}\label{Gradient-1'}
	\nabla_{i}v=\frac{\nabla_{l}u\nabla_{il}u}{|\nabla u|^{2}}-a\frac{\nabla_{i}\phi}{\phi}=0
\end{equation}
and
\begin{equation}\label{Gradient-1}
	F^{\tau}v_{t}\geq 0 \geq F^{ii}\nabla_{ii}v.
\end{equation}
Thus, by \eqref{Gradient-1'} and \eqref{Gradient-1}, we have
\begin{equation}\label{Gradient-x}
	\begin{aligned}
		0\geq\,&  F^{ii}\nabla_{ii}v-F^{\tau}v_{t}\\
		=   \,&  F^{ii}\nabla_{ii}(\log|\nabla u|)-F^{\tau}(\log |\nabla u|)_{t}
		-aF^{ii}\nabla_{ii}\log \phi+aF^{\tau}(\log \phi)_{t}\\
		=   \,&  \frac{1}{|\nabla u|^{2}}F^{ii}\nabla_{il}u\nabla_{il}u
		+\frac{\nabla_{l}u}{|\nabla u|^{2}}\Big(F^{ii}\nabla_{iil}u-F^{\tau}\nabla_{l}u_{t}\Big)\\
		\,& +\frac{a-2a^{2}}{\phi^{2}}F^{ii}(\nabla_{i}\phi)^{2}-\frac{a}{\phi}F^{ii}\nabla_{ii}\phi.
	\end{aligned}
\end{equation}
Differentiating both sides of Eq \eqref{eqn} with respect to $x$, we obtain, at $(x_{0},t_{0})$,
\begin{equation}
	\label{deqn-1}
	F^{ii} \nabla_{k} U_{ii}
	- F^{\tau}\nabla_k u_t = \psi_{k} +  \psi_{p_j} \nabla_{kj} u
\end{equation}
for all $k = 1, \ldots, n$.

Let $\phi=-u+\sup_{\bM_{T}}u+1$.
Note that, at $(x_{0}, t_{0})$, $\nabla_{ij}u=\nabla_{ij}u$ and
\begin{equation}\label{hess-A70}
	\nabla_{ijk}u-\nabla_{jik}u=R^{l}_{kij}\nabla_{l}u.
\end{equation}
By \eqref{G1}, \eqref{G2}, \eqref{Gradient-1'}, \eqref{deqn-1} and \eqref{hess-A70}, we have
\begin{equation}\label{Gradient-3}
	\begin{aligned}
		\frac{\nabla_{l}u}{|\nabla u|^{2}}\Big(F^{ii}\nabla_{iil}u-F^{\tau}\nabla_{l}u_{t}\Big)
		=   \,& \frac{\nabla_{l}u}{|\nabla u|^{2}}F^{ii}(\nabla_{lii}u-R^{k}_{iil}\nabla_{k}u-F^{\tau}\nabla_{l}u_{t})\\
		\geq\,& \frac{\nabla_{l}u}{|\nabla u|^{2}}F^{ii}(\nabla_{l}U_{ii}-\nabla_{l}(A^{ii})-F^{\tau}\nabla_{l}u_{t})-C\\
		\geq\,& -C(1+|\nabla u|^{\gamma-2})(1+\sum F^{ii}).
	\end{aligned}
\end{equation}
Therefore, by substituting \eqref{Gradient-3} into \eqref{Gradient-x}, we have
\begin{equation}\label{Gradient-x'}
	\begin{aligned}
		0\geq\,&  \frac{1}{|\nabla u|^{2}}F^{ii}\nabla_{il}u\nabla_{il}u
		+\frac{a-2a^{2}}{\phi^{2}}F^{ii}(\nabla_{i}u)^{2}+\frac{a}{\phi}F^{ii}\nabla_{ii}u\\
		\,& -C(1+|\nabla u|^{\gamma-2})(1+\sum F^{ii}).
	\end{aligned}
\end{equation}
Notice that
$$
\frac{1}{|\nabla u|^{2}}F^{ii}\nabla_{ii}u\nabla_{ii}u+\frac{a}{\phi}F^{ii}\nabla_{ii}u\geq -\frac{a^{2}|\nabla u|^{2}}{4\phi^{2}}\sum F^{ii}.
$$
It follows from \eqref{Gradient-x'} that
\begin{equation}\label{Gradient-target}
	\begin{aligned}
		0\geq\,& \frac{a-2a^{2}}{\phi^{2}}F^{ii}(\nabla_{i}u)^{2}-\frac{a^{2}|\nabla u|^{2}}{4\phi^{2}}\sum F^{ii} \\
		\,& -C(1+|\nabla u|^{\gamma-2})(1+\sum F^{ii}).
	\end{aligned}
\end{equation}

Without loss of generality we may consider $\nabla_{1} u (x_{0}, t_0) \geq \frac{1}{n} |\nabla u (x_{0}, t_0)| > 0$.
Recall that $U_{ij} (x_0, t_0)$ is diagonal. By \eqref{G3} and \eqref{Gradient-1'},  we have
\begin{equation}
	\label{Gradient-0'}
	\begin{aligned}
		U_{11} =    \,& -\frac{a}{\phi}|\nabla u|^2+A^{11} + \frac{\sum_{l \geq 2} \nabla_l u A^{1l}}{\nabla_1 u}\\
		\leq \,& -\frac{a}{\phi}|\nabla u|^2+ C (1 +|\nabla u|^{\gamma_{1}}) < 0
	\end{aligned}
\end{equation}
provided $|\nabla u|$ is sufficiently large.
The appearance of $A^{1l}$ in the first line is due to the diagonality of $\{U_{ij}\}$.
Therefore, by \eqref{Grd1},
$$ f_{1} \geq \nu_{0} \Big(1 + \sum^n_{i = 1} f_{i}+F^{\tau}\Big) $$
and a bound $|\nabla u (x_0, t_0)| \leq C_{3}$ follows from \eqref{Gradient-target} by choosing $a$ sufficiently small such that
$$
\frac{a-2a^{2}}{\phi^{2}}\cdot\frac{\nu_{0}}{n}-\frac{a^{2}}{4\phi^{2}}\geq c_{1}>0
$$
holds for some uniform constant $c_{1}$.

\end{proof}

\begin{remark}
This assumptions follow from \cite{GuanJiao.2016} and \cite{Jiao.2015}.
\eqref{G3} with $\gamma_{1}\in (0,2)$ is more of a technical condition here.
Actually, it will be better to obtain gradient estimates with quadratic growth conditions, i.e $\gamma_{1}=2$, see examples in \cite{TrudingerWang2008}.
The reason why we need \eqref{G3} is the regular assumption of $A$ which make us can not use barrier $\eta=e^{K(\ul u-u)}$ in gradient estimates.
From the proof of Lemma~\ref{lemma2.1} you can see the proof of the barrier is based on the gradient estimates.
This requirement also occurs in Theorem 1.3 (ii) in \cite{jiangtrudinger.2018}.

\eqref{Grd1} is a natural assumption satisfied by many operators such as the k-Hessian operator $\sigma_{k}^{\frac{1}{k}}$.
It is commonly used in deriving gradient estimate, for example in \cite{GuanSpruck.1991}.
\end{remark}

\section{The estimates for \texorpdfstring{$|u_{t}|$}{|ut|} }

In this section, we derive the estimates for $|u_{t}|$.
\begin{thm}
\label{ut_thm}
Suppose that \eqref{f1}--\eqref{f2}, \eqref{sub} and \eqref{ut-condition} hold, $A=A(x,t,\nabla u)$ and $\psi=\psi(x,t,\nabla u)$.
Let $u \in C^{3}(\bM_{T})$ be an admissible solution of \eqref{eqn}-\eqref{eqn-bd} in $M_{T}$.
Then there exists a positive constant $C_{2}$ depending on $|u|_{C^{1}(\bM_{T})}$,
$|\ul u|_{C^{2}(\bM_{T})}$, $|\psi|_{C^{2}(\bM_{T})}$ and other known data such that
\begin{equation}\label{|ut|_est}
	\sup_{\bM_{T}} |u_{t}| \leq C_{4}(1+\sup_{\mathcal{P} M_{T}}|u_{t}|).
\end{equation}
\end{thm}
\begin{proof}
We first show that
\begin{equation}\label{ut1}
	\sup_{\bM_{T}} (-u_{t}) \leq C_{4}(1+\sup_{\mathcal{P} M_{T}}|u_{t}|)
\end{equation}
for which we set
\[
W=\sup_{\bM_{T}}(-u_{t})e^{\phi},
\]
where $\phi$ is a positive function to be chosen.

We may assume that $W$ is attained at $(x_{0}, t_{0}) \in \bM_{T}-\mathcal{P}M_{T}$. As in the proof of
Theorem~\ref{p-th0}, we choose an orthonormal local frame $e_{1}, \ldots, e_{n}$ about $x_{0}$ such that
$\nabla_{e_{i}} e_{j}=0$ and $\{U_{ij}(x_{0}, t_{0})\}$ is diagonal.
We may assume $-u_{t}(x_{0}, t_{0})>0$.
Define $v=\log (-u_{t})+\phi$.
At $(x_{0}, t_{0})$, where the function $v$ achieves its maximum, we have, for $i=1, \ldots n$,
\begin{equation}
	\label{ut-1'}
	\nabla_{i}v=\frac{\nabla_i u_{t}}{u_{t}} +\nabla_{i}\phi  = 0
\end{equation}
and
\begin{equation}
	\label{ut-1}
	F^{\tau}v_{t}\geq 0 \geq F^{ii}\nabla_{ii}v=F^{ij}\nabla_{ii}v+(F^{ij}A^{ij}_{p_{k}}-\psi_{p_{k}})\nabla_{k}v.
\end{equation}
Thus, by \eqref{ut-1'} and \eqref{ut-1}, we have
\begin{equation}\label{ut-x}
	\begin{aligned}
		0\geq\,&  F^{ii}\nabla_{ii}v-F^{\tau}v_{t}+(F^{ij}A^{ij}_{p_{k}}-\psi_{p_{k}})\nabla_{k}v\\
		=   \,&  F^{ii}\nabla_{ii}\log(-u_{t})-F^{\tau}(\log(-u_{t}))_{t}
		+F^{ii}\nabla_{ii}\phi-F^{\tau}\phi_{t}\\
		\,& +(F^{ij}A^{ij}_{p_{k}}-\psi_{p_{k}})\nabla_{k}(\log(-u_{t})+\phi)\\
		=   \,&  \frac{1}{u_{t}}\Big(F^{ii}\nabla_{ii}u_{t}-F^{\tau}u_{tt}+(F^{ij}A^{ij}_{p_{k}}-\psi_{p_{k}})\nabla_{k}u_{t}\Big)\\
		\,&  +\mathcal{L}\phi-F^{ii}(\nabla_{i} \phi)^{2}.
	\end{aligned}
\end{equation}

By differentiating equation \eqref{eqn} with respect to $t$, we get
\begin{equation}\label{deqn-t}
	F^{ii}(U_{ii})_{t}-F^{\tau}u_{tt}=\psi_{t}+\psi_{p_{k}}(\nabla_{k}u)_{t}.
\end{equation}
It follows from \eqref{ut-x} and \eqref{deqn-t} that
\begin{equation}\label{ut-target}
	\begin{aligned}
		0\geq \,& \frac{1}{u_{t}}((\psi_{t}-F^{ii}A^{ii}_{t})-F^{ii}(\nabla_{i}\phi)^{2}+\mathcal{L}\phi\\
		\geq \,& \frac{C}{u_{t}}(1+\sum F^{ii})-F^{ii}(\nabla_{i}\phi)^{2}+\mathcal{L}\phi.
	\end{aligned}
\end{equation}
Fix a positive constant $\alpha\in(0,1)$ and let $\phi=\frac{\delta^{1+\alpha}}{2}|\nabla u|^{2}+\delta u+b\eta$,
where $\eta=e^{K(\ul u-u)}$ as in Lemma~\ref{lemma2.1} and $\delta\ll b \ll 1$ are positive constants to be determined.
By straightforward calculations, we have
\[
\nabla_{i}\phi=\delta^{1+\alpha}\sum_{k}\nabla_{k}u\nabla_{ik}u+\delta\nabla_{i}u+b\nabla_{i}\eta,
\]
\[
\phi_{t}=\delta^{1+\alpha}\sum_{k}\nabla_{k}u(\nabla_{k}u)_{t}+\delta u_{t}+b\eta_{t},
\]
\[
\nabla_{ii}\phi=\delta^{1+\alpha}\sum_{k}(\nabla_{ik}u)^{2}+\delta^{1+\alpha}\sum_{k}\nabla_{k}u\nabla_{iik}u
+\delta\nabla_{ii}u+b\nabla_{ii}\eta.
\]
It follows that
\begin{equation}\label{ut-2}
	\begin{aligned}
		\mathcal{L}\phi\geq\,& \delta^{1+\alpha}\nabla_{k}u \Big(F^{ii}\nabla_{iik}u-F^{\tau}(\nabla_{k}u)_{t}
		+F^{ij}A^{ij}_{p_{l}}\nabla_{kl}u-\psi_{p_{l}}\nabla_{kl}u \Big)\\
		\,& +\frac{\delta^{1+\alpha}}{2}F^{ii}U_{ii}^{2}
		-C\delta^{1+\alpha}\sum F^{ii}+\delta\mathcal{L}u+b\mathcal{L}\eta\\
		\geq\,& -C\delta^{1+\alpha}\Big(1+\sum F^{ii} \Big)+\frac{\delta^{1+\alpha}}{2}F^{ii}U_{ii}^{2}
		+\delta\mathcal{L}u+b\mathcal{L}\eta
	\end{aligned}
\end{equation}
and
\begin{equation}\label{ut-3}
	(\nabla_{i}\phi)^{2}\leq C\delta^{2(1+\alpha)}U_{ii}^{2}+Cb^{2}
\end{equation}
since $b \gg \delta$.
Thus, \eqref{ut-target} becomes, by \eqref{ut-2} and \eqref{ut-3},
\begin{equation}\label{ut-main}
	b\mathcal{L}\eta+\frac{\delta^{1+\alpha}}{4}F^{ii}U_{ii}^{2}+\delta\mathcal{L}u
	\leq -\frac{C}{u_{t}}(1+\sum F^{ii})+C\delta^{1+\alpha}\Big(1+\sum F^{ii} \Big)+Cb^{2}\sum F^{ii}.
\end{equation}
We first consider case (i): $|\nu_{\mu}-\nu_{\lambda}|\geq \beta$.
Note that
\[
\delta F^{ii}U_{ii}\geq -\frac{\delta^{1+\alpha}}{4}F^{ii}U_{ii}^{2}-\delta^{1-\alpha}\sum F^{ii}.
\]
It follows from that
\begin{equation}\label{ut-1-mian}
	\begin{aligned}
		\frac{\delta^{1+\alpha}}{4}F^{ii}U_{ii}^{2}+\delta\mathcal{L}u
		\geq\,& -C\delta(1+\sum F^{ii})+\frac{\delta^{1+\alpha}}{4}F^{ii}U_{ii}^{2}\\
		\,& +\delta F^{ii}U_{ii}-\delta F^{\tau}u_{t}\\
		\geq\,& -C\delta^{1-\alpha}(1+\sum F^{ii})
	\end{aligned}
\end{equation}
since $u_{t}(x_{0},t_{0})<0$.
Therefore, by \eqref{ut-main} and \eqref{ut-1-mian}, we have
\begin{equation}\label{ut-1-end}
	b\mathcal{L}\eta
	\leq -\frac{C}{u_{t}}(1+\sum F^{ii})+C\delta^{1-\alpha}\Big(1+\sum F^{ii} \Big)+Cb^{2}\sum F^{ii}.
\end{equation}
Choosing $b$ and $\delta$ such that $b\epsilon_{0}-C\delta^{1-\alpha}-Cb^2\geq b_{1}>0$ for a positive constant $b_{1}$, then a upper bound of $-u_{t}(x_{0},t_{0})$ derived by \eqref{lem2.1}.

Case (ii): $|\nu_{\mu}-\nu_{\lambda}|< \beta$. We see that \eqref{case 2} holds.
Note that
\[
\frac{\delta^{1+\alpha}}{8}F^{ii}U_{ii}^{2}+\delta F^{ii}U_{ii}\geq -2\delta^{1-\alpha}\sum F^{ii}
\]
and
\begin{equation}\label{Ln}
	\begin{aligned}
		\mathcal{L}e^{K(\ul u-u)}=\,&Ke^{K(\ul u-u)}[\mathcal{L}(\ul u-u)+KF^{ij}D_{i}(\ul u-u)D_{j}(\ul u-u)]\\
		\geq\,& Ke^{K(\ul u-u)}\Big\{-\frac{1}{2}F^{ij}A^{ij}_{p_{k},p_{l}}(x,t,\hat{p})D_{k}(\ul u-u)D_{l}(\ul u-u)\\
		\,&+KF^{ij}D_{i}(\ul u-u)D_{j}(\ul u-u)\Big\}\\
		\geq\,& -C\sum F^{ii}
	\end{aligned}
\end{equation}
by the concavity of $F$ and $\psi$, where $C$ depends on $|u|_{C^{1}(\bM_{T})}$ and other known data.
We have, by \eqref{ut-main},
\begin{equation}\label{ut-2-main}
	\begin{aligned}
		\frac{\delta^{1+\alpha}}{8}F^{ii}U_{ii}^{2}-\delta F^{\tau}u_{t}
		\leq\,& -\frac{C}{u_{t}}(1+\sum F^{ii})+C\delta\Big(1+\sum F^{ii} \Big)\\
		\,&+C(\delta^{1-\alpha}+b+b^{2})\sum F^{ii}\\
		\leq\,& -\frac{C}{u_{t}}(1+\sum F^{ii})+C\delta^{1-\alpha}+C\sum F^{ii}.
	\end{aligned}
\end{equation}
Recalling that $u_{t}<0$, we get
\[
F^{ii}U_{ii}-F^{\tau}u_{t}\geq u_{t}\Big(\sum F^{ii}+F^{\tau}\Big)+\frac{1}{4u_{t}}\Big(F^{ii}U_{ii}^{2}+F^{\tau}u_{t}^{2}\Big).
\]
Therefore, by the concavity of $f$, we have
\begin{equation}\label{ut-2-1}
	\begin{aligned}
		-u_{t}\Big(\sum F^{ii}+F^{\tau}\Big)
		\geq\,& f(-u_{t}\textbf{1})-f(\lambda(U),-u_{t})+F^{ii}U_{ii}-F^{\tau}u_{t}\\
		\geq\,& u_{t}\Big(\sum F^{ii}+F^{\tau}\Big)+\frac{1}{4u_{t}}\Big(F^{ii}U_{ii}^{2}+F^{\tau}u_{t}^{2}\Big)\\
		\,& +f(-u_{t}\textbf{1})-\psi[u],
	\end{aligned}
\end{equation}
where $\textbf{1}=(1, \ldots, 1)\in \mathbb{R}^{n+1}$.

Note that $\lim_{t\rightarrow \infty}f(t\textbf{1})=\sup_{\Gamma}f>\sup_{\bM_{T}}\psi[u]$.
It follows from \eqref{f1} that
\begin{equation}\label{ut-2-2}
	f(-u_{t}\textbf{1})-\psi[u]\geq f(-u_{t}\textbf{1})-\sup_{\bM_{T}}\psi[u]:=2b_{2}
\end{equation}
provided $-u_{t}(x_{0},t_{0})$ is big enough, where $b_{2}$ is a positive constant.
Therefore, by \eqref{ut-2-1} and \eqref{ut-2-2}, we have
\begin{equation}\label{ut-2-3}
	-u_{t}\Big(\sum F^{ii}+F^{\tau}\Big)\geq b_{2}+\frac{1}{8u_{t}}\Big(F^{ii}U_{ii}^{2}+F^{\tau}u_{t}^{2}\Big).
\end{equation}
It follows from \eqref{case 2} and \eqref{ut-2-3} that
\begin{equation}\label{ut-2-4}
	\begin{aligned}
		-F^{\tau}u_{t}\geq\,& -2\gamma_{0}u_{t}\Big(\sum F^{ii}+F^{\tau}\Big)\\
		\geq\,& -\gamma_{0}u_{t}\Big(\sum F^{ii}+F^{\tau}\Big)+\gamma_{0}b_{2}
		+\frac{\gamma_{0}}{8u_{t}}\Big(F^{ii}U_{ii}^{2}+F^{\tau}u_{t}^{2}\Big)\\
		\geq\,& -\gamma_{0}u_{t}\sum F^{ii}+\gamma_{0}b_{2}
		+\frac{\gamma_{0}}{8u_{t}}F^{ii}U_{ii}^{2},
	\end{aligned}
\end{equation}
where $\gamma_{0}:=\frac{\beta}{2\sqrt{n+1}}>0$.

Without loss of generality, we suppose $-u_{t}\geq \gamma_{0}\delta^{-\alpha}$ for fixed $\delta$.
Substituting \eqref{ut-2-4} in \eqref{ut-2-main} we derive
\begin{equation}\label{ut-2-end}
	(-\delta\gamma_{0}u_{t}-C)\sum F^{ii}+\delta\gamma_{0} b_{2}-C\delta^{1-\alpha}\leq -\frac{C}{u_{t}}(1+\sum F^{ii}).
\end{equation}
By \eqref{ut-condition}, we see that $b_{2}$ can be sufficiently large,
then a bound is derived from \eqref{ut-2-end} and therefore \eqref{ut1} holds.

Similarly, we can show
\begin{equation}\label{ut2}
	\sup_{\bM_{T}} u_{t} \leq C_{4}(1+\sup_{\mathcal{P} M_{T}}|u_{t}|)
\end{equation}
by letting
\[
\phi=\frac{\delta^{1+\alpha}}{2}|\nabla u|^{2}-\delta u+b(\ul u-u).
\]

Combining \eqref{ut1} and \eqref{ut2}, the proof is finished.
\end{proof}

\begin{remark}
If $\ul u$ is a strict subsolution, then Theorem~\ref{ut_thm} follows without \eqref{ut-condition}.
In face, in this case we have \eqref{lem2.1} holds without classification.
Let $W=\sup_{\bM_{T}}|u_{t}|e^{a\phi}$ and $\phi=\eta$ in Lemma~\ref{lem2.1}, the theorem will be proved easily.
\end{remark}

By \eqref{comp1} and \eqref{comp2} we can the short time existence as Theorem~15.9 in \cite{Lieberman.1996}.
So without of loss of generality, we may assume that $\varphi$ is defined on $M\times [0, t_{0}]$ for some small constant $t_{0}>0$ and
\begin{equation}\label{comp3}
f(\lambda(\nabla^{2}\varphi(x,0)+A[\varphi]), -\varphi_{t}(x,0))=\psi[\varphi]\;\;\forall x \in \bM.
\end{equation}
Since that $u_{t}=\varphi_{t}$ on $SM_{T}$ and \eqref{comp3}, we can obtain the estimate
\begin{equation}\label{ut}
\sup_{\bM_{T}}|u_{t}|\leq C_{5}.
\end{equation}

\section{Global estimates for second derivatives}

In this section, we derive the global estimates for the second order derivatives. In particular, we prove the following maximum principle.
\begin{thm}
Let $u\in C^{4}(\bM_{T})$ be an admissible solution of \eqref{eqn} in $M_{T}$.
Suppose that \eqref{f1}--\eqref{f2} and \eqref{sub}--\eqref{A3w} hold. Then
\begin{equation}\label{Glb_est}
	\sup_{\bM_{T}}|\nabla^{2} u|\leq C_{1}(1+\sup_{\mathcal{P}M_{T}}|\nabla^{2}u|),
\end{equation}
where $C_{1}>0$ depends on $|u|_{C^{1}(\bM_{T})}$, $|\ul u|_{C^{1}(\bM_{T})}$, $|u_{t}|_{C^{0}(\bM_{T})}$,
$|\psi|_{C^{2}(\bM_{T})}$ and other known data.
\end{thm}
\begin{proof}
Set
\[W = \max_{(x,t) \in \bar{M_T}} \max_{\xi \in T_x M, |\xi| = 1}
(\nabla_{\xi\xi} u + A^{\xi \xi} (x, t, \nabla u) )e^\phi,\]
as in \cite{GuanJiao.2016}, where $\phi$ is a function to be determined. It suffices to estimate $W$.
We may assume $W$ is achieved at $(x_{0}, t_{0}) \in \bM_T - \mathcal{P} M_T$.
Choose a smooth orthonormal local frame $e_{1}, \ldots, e_{n}$ about $x_{0}$
such that $\nabla_i e_j = 0$,  and
$\{U_{ij}\}$ is diagonal at $(x_0, t_0)$.
We assume $U_{11} (x_0, t_0) \geq \ldots \geq U_{nn} (x_0, t_0)$ and,
without loss of generality, we assume $U_{11}>1$.

Define $v=\log U_{11} + \phi$.
At $(x_{0}, t_{0})$, where the function
$v$ attains its maximum, we have, for each $i = 1, \ldots, n$,
\begin{equation}
	\label{2nd-gl-1'}
	\nabla_{i}v=\frac{\nabla_i U_{11}}{U_{11}} + \nabla_i \phi = 0
\end{equation}
and
\begin{equation}
	\label{2nd-gl-1}
	F^{\tau}v_{t}\geq 0 \geq F^{ii}\nabla_{ii}v.
\end{equation}
Thus, by \eqref{2nd-gl-1}, we have
\begin{equation}\label{2nd-gl-x}
	\begin{aligned}
		0\geq\,&  F^{ii}\nabla_{ii}v-F^{\tau}v_{t}\\
		=   \,&  F^{ii}\nabla_{ii}(\log U_{11})-F^{\tau}(\log U_{11})_{t}+F^{ii}\nabla_{ii}\phi-F^{\tau}\phi_{t}\\
		=   \,& -\frac{1}{U_{11}^{2}}F^{ii}\nabla_{i}U_{11}^{2}
		+\frac{1}{U_{11}}\Big(F^{ii}\nabla_{ii}U_{11}-F^{\tau}(U_{11})_{t}\Big)\\
		\,& +F^{ii}\nabla_{ii}\phi-F^{\tau}\phi_{t}.
	\end{aligned}
\end{equation}

Differentiating Eq (\ref{eqn}) twice, we obtain, by \eqref{A2}, \eqref{deqn-1}. \eqref{hess-A70} and \eqref{2nd-gl-1'},
\begin{equation}
	\label{deqn-2}
	\begin{aligned}
		F^{ii} \nabla_{11} U_{ii}\,&+F^{ij,kl} \nabla_{1} U_{ij} \nabla_{1} U_{kl}-2F^{ij,\tau}\nabla_{1}U_{ij}\nabla_{1}u_{t} \\
		\,&+F^{\tau, \tau}(\nabla_{1}u_{t})^{2}- F^{\tau}\nabla_{11} u_t \\
		\geq\,&-CU_{11}+\psi_{p_{k}p_{l}}\nabla_{1k}u\nabla_{1l}u+\psi_{p_{k}}\nabla_{11l}u\\
		\geq\,&-CU_{11}-U_{11}\psi_{p_{k}}\nabla_{k}\phi.
	\end{aligned}
\end{equation}

Note that the regular condition of $A$ means $A^{ii}_{p_{1}p_{1}}\leq 0$ for $i\neq 1$.
Therefore by \eqref{deqn-1} and \eqref{2nd-gl-1'}, we have
\begin{equation}\label{2nd-gl-A}
	\begin{aligned}
		F^{ii}(\nabla_{ii}A^{11}-\nabla_{11}A^{ii})
		\geq\,& F^{ii}(A^{11}_{p_{k}}\nabla_{iik}u-A^{ii}_{p_{k}}\nabla_{11k}u)-CU_{11}\sum F^{ii}\\
		\,&+F^{ii}(A^{11}_{p_{i}p_{i}}U_{ii}^{2}-A^{ii}_{p_{1}p_{1}}U_{11}^{2})\\
		\geq\,& U_{11}F^{ii}A^{ii}_{p_{k}}\nabla_{k}\phi+F^{\tau}A^{11}_{p_{k}}\nabla_{k}u_{t}-CU_{11}\sum F^{ii}\\
		\,&-CU_{11}-C\sum_{i\geq 2}F^{ii}U_{ii}^{2}.
	\end{aligned}
\end{equation}
Note that
$$
\begin{aligned}
	\nabla_{ijkl} v - \nabla_{klij} v
	=& R^m_{ljk} \nabla_{im} v  + \nabla_i R^m_{ljk} \nabla_m v
	+ R^m_{lik} \nabla_{jm} v \\
	& + R^m_{jik} \nabla_{lm} v
	+ R^m_{jil} \nabla_{km} v + \nabla_k R^m_{jil} \nabla_m v.
\end{aligned}
$$
Thus we have
\begin{equation}\label{2nd-gl-4th}
	\nabla_{ii}U_{11}\geq\nabla_{11}U_{ii}+\nabla_{ii}A^{11}-\nabla_{11}A^{ii}-CU_{11}.
\end{equation}
It follows from \eqref{deqn-2}, \eqref{2nd-gl-A} and \eqref{2nd-gl-4th} that
\begin{equation}\label{2nd-gl-2}
	\begin{aligned}
		F^{ii}\nabla_{ii}U_{11}-F^{\tau}(U_{11})_{t}
		\geq\,& F^{ii}\nabla_{11}U_{ii}-F^{\tau}\nabla_{11}u_{t}-CU_{11}\sum F^{ii}\\
		\,&-F^{ii}(\nabla_{ii}A^{11}-\nabla_{11}A^{ii})-F^{\tau}(A^{11})_{t}\\
		\geq\,&-F^{ij,kl} \nabla_{1} U_{ij} \nabla_{1} U_{kl}-2F^{ij,\tau}\nabla_{1}U_{ij}\nabla_{1}u_{t}\\
		\,&+F^{\tau, \tau}(\nabla_{1}u_{t})^{2}+U_{11}(F^{ii}A^{ii}_{p_{k}}-\psi_{p_{k}})\nabla_{k}\phi\\
		\,&-C\sum_{i\geq 2}F^{ii}U_{ii}^{2}-CU_{11}(1+\sum F^{ii}).
	\end{aligned}
\end{equation}

Thus, by \eqref{2nd-gl-x} and \eqref{2nd-gl-2}, we have, at $(x_{0}, t_{0})$,
\begin{equation}
	\label{2nd-gl-target}
	\mathcal{L}\phi \leq \frac{C}{U_{11}}\sum_{i\geq 2}F^{ii}U_{ii}^{2}+C(1+\sum F^{ii})+E,
\end{equation}
where
\[
E=\frac{1}{U_{11}^{2}}F^{ii}(\nabla_{i}U_{11})^{2}+\frac{1}{U_{11}}(F^{ij, kl}\nabla_{1}U_{ij}\nabla_{1}U_{kl}
-2F^{ij, \tau}\nabla_{1}U_{ij}\nabla_{1}u_{t}+F^{\tau,\tau}(\nabla_{1}u_{t})^{2}).
\]

Let $\eta=e^{K(\ul u-u)}$. Define
\[ \phi = \frac{\delta |\nabla u|^2}{2} + b \eta, \]
where $b$ and $\delta$ are undetermined constants such that $0 < \delta < 1 \leq b$.
We find, at $(x_0, t_0)$,
\begin{equation}
	\label{dphi-1}
	\nabla_{i} \phi
	= \delta \nabla_{j} u \nabla_{ij} u + b \nabla_{i} \eta
	= \delta \nabla_i u U_{ii} - \delta \nabla_{j} u A^{ij} + b \nabla_{i} \eta,
\end{equation}
\begin{equation}
	\label{dphi-t}
	\phi_t = \delta \nabla_{j} u (\nabla_{j} u)_t + b \eta_t,
\end{equation}
\begin{equation}
	\label{dphi-2}
	\nabla_{ii} \phi
	\geq \frac{\delta}{2} U_{ii}^2 - C \delta + \delta \nabla_{j} u \nabla_{iij} u
	+ b \nabla_{ii} \eta.
\end{equation}
From \eqref{hess-A70} and \eqref{deqn-1}, we derive
\begin{equation}
	\label{ps-gs7}
	\begin{aligned}
		F^{ii} \nabla_{j} u \nabla_{iij} u \geq \,& F^{ii} \nabla_{j} u (\nabla_{j} U_{ii} - \nabla_{j} A^{ii})
		- C |\nabla u|^2 \sum F^{ii} \\
		\geq \,& (\psi_{p_{k}} - F^{ii} A^{ii}_{p_k}) \nabla_{j} u \nabla_{jk} u
		+ F^{\tau}\nabla_j u \nabla_j (u_t) \\
		\,& - C (1 + \sum F^{ii}).
	\end{aligned}
\end{equation}
Therefore,
\begin{equation}
	\label{ps-S100}
	\begin{aligned}
		\mathcal{L} \phi \geq b \mathcal{L} \eta + \frac{\delta}{2} F^{ii} U_{ii}^2
		- C \delta(1+\sum F^{ii}).
	\end{aligned}
\end{equation}
Next, by \eqref{dphi-1} we get
\begin{equation}
	\label{bs-gs3.5}
	\begin{aligned}
		(\nabla_{i} \phi)^2
		\leq C \delta^2 (1 + U_{ii}^2) + 2 b^2 (\nabla_{i} (\ul u - u))^2
		\leq C \delta^2 U_{ii}^2 + C b^2.
	\end{aligned}
\end{equation}

Now we estimate $E$ as in \cite{Guan.2014} and \cite{GuanJiao.2015} (see \cite{BaoDongJiao.2016} for details). Let
\[  \begin{aligned}
	J \,& = \{i: U_{ii} \leq - s U_{11}\}, \;\;
	K = \{i:  U_{ii} > - s U_{11} \},
\end{aligned} \]
where $0 < s \leq 1/3$ is a fixed number. Using an
inequality of Andrews~\cite{Andrews.1994}
and Gerhardt~\cite{Gerhardt.1996}, we have, by \eqref{bs-gs3.5},
\begin{equation}
	\label{gj-S130}
	\begin{aligned}
		- F^{ij, kl} \nabla_1 U_{ij} \nabla_1 U_{kl}
		\geq \,& \sum_{i \neq j} \frac{F^{ii} - F^{jj}}{U_{jj} - U_{ii}}
		(\nabla_1 U_{ij})^2 \\
		\geq \,& 2 \sum_{i \geq 2} \frac{F^{ii} - F^{11}}{U_{11} - U_{ii}}
		(\nabla_1 U_{i1})^2 \\
		\geq \,& \frac{2 (1-s)}{(1+s) U_{11}} \sum_{i \in K} (F^{ii} - F^{11})
		((\nabla_i U_{11})^2 - C U_{11}^2/s).
	\end{aligned}
\end{equation}
Thus, we obtain
\begin{equation}
	\label{gj-S140}
	\begin{aligned}
		E
		\leq \,& \frac{1}{U_{11}^2} \sum_{i \in J} F^{ii} (\nabla_i U_{11})^2
		+ C \sum_{i \in K} F^{ii}
		+ \frac{C F^{11}}{U_{11}^2}  \sum_{i \in K} (\nabla_i U_{11})^2  \\
		\leq \,& \sum_{i \in J} F^{ii} (\nabla_i \phi)^2
		+  C \sum F^{ii} + C F^{11} \sum (\nabla_i \phi)^2 \\
		\leq \,& C b^2 \sum_{i \in J} F^{ii}  + C \delta^2 \sum F^{ii} U_{ii}^2
		+  C \sum F^{ii} + C (\delta^2 U_{11}^2 + b^2) F^{11}.
	\end{aligned}
\end{equation}
Therefore, by \eqref{2nd-gl-target}, \eqref{ps-S100}, \eqref{bs-gs3.5} and \eqref{gj-S140}, we have
\begin{equation}
	\label{ps-S150}
	\begin{aligned}
		b \mathcal{L} \eta
		\leq \,& \Big(C \delta^2 + \frac{C}{U_{11}} - \frac{\delta}{2}\Big) F^{ii} U_{ii}^2
		+ C b^2 \sum_{i \in J} F^{ii}\\
		\,&  + C (\delta^2 U_{11}^2 + b^2) F^{11}  + C (1 + \sum F^{ii}).
	\end{aligned}
\end{equation}

Case (i): $|\nu_{\mu}-\nu_{\lambda}|\geq \beta$. It follows from \eqref{lem2.1} and \eqref{ps-S150} that
\[
\begin{aligned}
	( b \varepsilon-C) (1 + \sum F^{ii}) \leq \,& \Big(C \delta^2 + \frac{C}{U_{11}} - \frac{\delta}{2}\Big) F^{ii} U_{ii}^2+C b^2 \sum_{i \in J} F^{ii} \\
	\,&  + C (\delta^2 U_{11}^2 + b^2) F^{11} .
\end{aligned}
\]
Choosing $b$ sufficiently large such that $b \varepsilon - C \geq \frac{b \varepsilon}{2}$, we have
\[
\begin{aligned}
	\frac{b \varepsilon}{2} (1 + \sum F^{ii}) \leq \,& \Big(C \delta^2 + \frac{C}{U_{11}} - \frac{\delta}{2}\Big) F^{ii} U_{ii}^2 +C b^2 \sum_{i \in J} F^{ii}\\
	\,&  + C (\delta^2 U_{11}^2 + b^2) F^{11} .
\end{aligned}
\]
and we can get a bound $U_{11} (x_0, t_0) \leq C$ by choosing $\delta$ sufficiently small since $|U_{ii}| \geq s U_{11}$ for $i \in J$.
Thus we derive a bound of $U_{11}(x_{0},t_{0})$ and therefore \eqref{Glb_est} holds.

Case (ii): $|\nu_{\mu}-\nu_{\lambda}|< \beta$. For every fixed $C>0$, choosing $\delta$ sufficiently small such that $\frac{\delta}{4}- C\delta^{2}\geq \delta_{0}>0$.
Without loss of generality, suppose $U_{11}\geq \frac{C}{\delta_{0}}$ for otherwise we are done. Then \eqref{ps-S150} becomes
\begin{equation}\label{g-case-2-main}
	\begin{aligned}
		b \mathcal{L} \eta+ \frac{\delta}{4} F^{ii} U_{ii}^2
		\leq \,&C b^2 \sum_{i \in J} F^{ii}  + C (\delta^2 U_{11}^2 + b^2) F^{11}  + C (1 + \sum F^{ii}).
	\end{aligned}
\end{equation}

Next, let $\hat{\lambda}:=\lambda(U(x_{0},t_{0}))$.
In the view of \eqref{ut-2-1}--\eqref{ut-2-3}, we have
\begin{equation}\label{g-case-2-1}
	|\hat{\lambda}|\Big(\sum F^{ii}+F^{\tau}\Big)\geq b_{3},
\end{equation}
where $b_{3}:=\frac{1}{2} \Big(f(|\hat{\lambda}|\textbf{1})-\sup_{\bM_{T}}\psi[u]\Big)>0$ provided $|\hat{\lambda}|$ is large enough.
By \eqref{case 2} and \eqref{g-case-2-1}, we have
$$
\frac{\delta}{4}F^{ii}U_{ii}^{2}\geq 2c_{2}|\hat{\lambda}|^{2}\Big(\sum F^{ii}+F^{\tau}\Big)\geq c_{2}|\hat{\lambda}|^{2}\Big(\sum F^{ii}+F^{\tau}\Big)+c_{2}b_{3}|\hat{\lambda}|,
$$
where $c_{2}=\frac{\delta\beta}{8\sqrt{n+1}}$. Therefore, it follows from \eqref{Ln} and \eqref{g-case-2-main} that
\begin{equation}\label{g-case-2-end}
	c_{2}|\hat{\lambda}|^{2}\Big(\sum F^{ii}+F^{\tau}\Big)+c_{2}b_{3}|\hat{\lambda}|
	\leq  C\delta^2 U_{11}^2F^{11}  + C(1 + \sum F^{ii}).
\end{equation}
Then a bound for $U_{11}$ is derived since $\delta\in (0,1)$ and $U_{11}\leq |\hat{\lambda}|$.

\end{proof}

\section{Boundary estimates for second derivatives}

In this section, we establish the estimates of second order derivatives on parabolic boundary
$\mathcal{P} M_T$. We may assume $\varphi \in C^4 (\bM_T)$. We shall establish the estimate
\begin{equation}\label{uii-s-1}
\max_{\mathcal{P}M_{T}}|\nabla^{2}u|\leq C_{2}
\end{equation}
for some positive constant $C_{2}$ depending on $|u|_{C^{1}{\bM_{T}}}$,
$|u_{t}|_{C^{0}{\bM_{T}}}$, $|\ul u|_{C^{2}{\bM_{T}}}$,
$|\psi|_{C^{4}{\bM_{T}}}$, and other known data.

Fix a point $(x_{0}, t_{0}) \in S M_T$. We shall choose
smooth orthonormal local frames $e_1, \ldots, e_n$ around $x_0$ such that
when restricted to $\partial M$, $e_n$ is the interior normal to $\partial M$ along the boundary when restricted
to $\partial M$.
Since $u - \ul{u} = 0$ on $S M_T$ we have
\begin{equation}
\label{hess-a200}
\nabla_{\alpha \beta} (u - \ul{u})
= -  \nabla_n (u - \ul{u}) \varPi (e_{\alpha}, e_{\beta}), \;\;
\forall \; 1 \leq \alpha, \beta < n \;\;
\mbox{on  $S M_T$},
\end{equation}
where 
$\varPi$ denotes the second fundamental form of $\partial M$.
Therefore,
\begin{equation}
\label{hess-E130}
|\nabla_{\alpha \beta} u| \leq  C,  \;\; \forall \; 1 \leq \alpha, \beta < n
\;\;\mbox{on} \;\; S M_T.
\end{equation}

Let $\rho (x)$ and $d(x)$ denote the distance from $x \in M$ to $x_{0}$ and $\partial M$ respectively
and set
\[M_{T}^{\delta} = \{X = (x, t) \in M \times (0,T]:
\rho (x) < \delta\}.\]

Now we shall use a perturbation method to obtain a strict subsolution from a non-strict one.
Let $s(x,t)=\ul u(x,t)+a(h(x)-1)$ and $S=\{\nabla_{ij}s+A[s]\}$, where $h(x)=e^{bd(x)}$, $a$ and $b$ are constants to be determined.
We wish to show $\tilde{M}=(F(S,-s_{t})-\psi[s])-(F(\ul U, -{\ul u}_{t})-\psi[\ul u])>0$ for some $a$ and $b$.
Note that $d$ is smooth near boundary and
\[
S_{ij}-{\ul U}_{ij}=ab^{2}h\nabla_{i}d\nabla_{j}d+abh\nabla_{ij}d+abhA^{ij}_{p_{k}}(x, t, \hat{p}_{1})\nabla_{k}d,
\]
where $\hat{p}_{1}=\nabla \ul u+\theta_{1} abh \nabla d$ for some $\theta_{1}\in (0,1)$.
Therefore, if $a$ is small enough for fixed $b$, $s$ is admissible since $\ul u$ is admissible and $\Gamma$ is open.
Let $F_{0}^{ij}=F^{ij}(\ul U, -{\ul u}_{t})$, there is a positive constant $c_{3}$ such that $F_{0}^{ij}\nabla_{i}d\nabla_{j}d\geq c_{3}>0$ since $|\nabla d(x)|\equiv 1$.
Thus, we derive
\[
\begin{aligned}
\tilde{M}\geq\,& F_{0}^{ij}(ab^{2}h\nabla_{i}d\nabla_{j}d+abh\nabla_{ij}d+abhA^{ij}_{p_{k}}(x, t, \tilde{p})\nabla_{k}d)\\
\,& -abh\psi_{p_{k}}(x, t, \hat{p}_{2})\nabla_{k}d\\
\geq\,& ab^{2}hc_{3}-abC
>       0,
\end{aligned}
\]
where  $b>C/c_{3}\geq C/hc_{3}$ and $\hat{p}_{2}=\nabla \ul u+\theta_{2} abh \nabla d$ for some $\theta_{2}\in (0,1)$.

Therefore a strict admissible subsolution with same boundary condition is derived near boundary and \eqref{lem2.1} holds without the assumption $|\nu_{\mu}-\nu_{\lambda}|\geq \beta$, see Remark~\ref{rk1}. For convenience, we still use $\ul u$ to denote the strict subsolution below.

For the mixed tangential-normal and pure normal second derivatives
at $(x_0, t_0)$, we shall use the following barrier function as in
\cite{Guan.2014},
\begin{equation}
\label{bd-0}
\varPsi = A_1 v
+ A_2 \rho^2 - A_3 \sum_{l < n} |\nabla_l (u - \varphi)|^2,
\end{equation}
where$$v=1-\eta=1-e^{K(\ul u-u)}$$
and $A_1$, $A_2$, $A_3$ are positive constants to be chosen.
By differentiating Eq \eqref{eqn} and
\[
\nabla_{ij}(\nabla_{k}u)=\nabla_{ijk}u+\Gamma_{ik}^{l}\nabla_{jl}u+\Gamma_{jk}^{l}\nabla_{il}u+\nabla_{\nabla_{ij}e_{k}}u,
\]
we obtain, by straightforward calculation,
\begin{equation}
\label{eq3-1}
\begin{aligned}
	\mathcal{L} (\nabla_k (u - \varphi))  \leq \,& C \Big(1 + \sum f_i |\lambda_i|
	+ \sum f_{i}+F^{\tau}\Big),
	\;\; \forall \; 1 \leq k \leq n,
\end{aligned}
\end{equation}
where $\lambda  = \lambda (\nabla^2 u + A [u])$.

The following lemma is crucial to construct barrier functions.
\begin{lem}
\label{barrier}
Suppose that \eqref{f1}--\eqref{f5} and \eqref{sub}--\eqref{A3w} hold. Then for
any positive constant $K_{1}$ there exist uniform positive constants $t,
\delta$ sufficiently small, and $A_1$, $A_2$, $A_3$
sufficiently large such that $\varPsi\geq K_{1}\rho^{2}$ in $\overline{M^{\delta}_{T}}$ and
\begin{equation}
	\label{eq3-5}
	\mathcal{L} \varPsi \leq - K_{1} \Big(1 + f_{i}|\lambda _{i}|+ \sum f_{i}+F^{\tau}\Big)  \;\;
	\mbox{in $\overline{M^{\delta}_{T}}$}.
\end{equation}
\end{lem}
\begin{proof}
First by Lemma~\ref{lemma2.1}, we have
\begin{equation}\label{uii-s-10}
	\mathcal{L}v \leq -\varepsilon\Big(1+\sum f_{i}+F^{\tau}\Big) \;\mbox{in} \, M_{T}^{\delta}.
\end{equation}

Similar to Proposition 2.19 of \cite{Guan.2014}, we can show that
\begin{equation}
	\label{jbd-4}
	\sum_{l < n} F^{ij} U_{il} U_{jl} \geq \frac{1}{2} \sum_{i \neq r} f_i \lambda _i^2,
\end{equation}
for some index $r$.
It follows that
\begin{equation}
	\label{bd-5}
	\begin{aligned}
		\sum_{l < n} \mathcal{L} |\nabla_l (u - \varphi)|^2
		\,& \geq  \sum_{l < n} F^{ij} U_{il} U_{jl} - C \Big(1 + \sum f_i |\lambda _i|+ \sum F^{ii}+F^{\tau}\Big)\\
		\,& \geq \frac{1}{2} \sum_{i \neq r} f_i \lambda _i^2- C \Big(1 + \sum f_i |\lambda _i|+ \sum F^{ii}+F^{\tau}\Big).
	\end{aligned}
\end{equation}

We first consider the case that $\lambda _{r} \geq 0$.
Notice that
$$
\begin{aligned}
	\mathcal{L}v=-Le^{K(\ul u-u)}\,&=-Ke^{K(\ul u-u)}[\mathcal{L}(\ul u-u)+KF^{ij}D_{i}(\ul u-u)D_{j}(\ul u-u)]\\
	\,&\geq a_{0}\sum f_{i}\lambda_{i}-C(1+\sum F^{ii}+F^{\tau}),
\end{aligned}
$$
where $a_{0}=\inf_{\mathcal{P} M_T}Ke^{K(\ul u-u)}$.

By  \eqref{uii-s-10}, \eqref{jbd-4} and \eqref{bd-5},
we obtain, for any $0 < B < A_{1}$,
\begin{equation}\label{uii-s-11}
	\begin{aligned}
		\mathcal{L} \varPsi \leq\,& (A_{1}+B)\mathcal{L}v-B\mathcal{L}v+CA_{2}\Big(1+\sum f_{i}+F^{\tau}\Big)
		-\frac{A_{3}}{2}\sum_{i \neq r}f_{i}\lambda _{i}^{2}\\
		\,&+CA_{3}\Big(1+f_{i}|\lambda _{i}|
		+\sum f_{i}+F^{\tau}\Big)\\
		\leq\,&-(A_{1}+B)\varepsilon\Big(1+\sum f_{i}+F^{\tau}\Big)-a_{0}Bf_{i}\lambda _{i}
		+CA_{3}f_{i}|\lambda _{i}|\\
		\,&-\frac{A_{3}}{2}\sum_{i \neq r}f_{i}\lambda _{i}^{2}+C(B+A_{2}+A_{3})\Big(1+\sum f_{i}+F^{\tau}\Big)\\
		\leq\,&-(A_{1}+B)\varepsilon\Big(1+\sum f_{i}+F^{\tau}\Big)
		+2a_{0}B\sum_{i \neq r}f_{i}|\lambda _{i}|-\frac{A_{3}}{2}\sum_{i \neq r}f_{i}\lambda _{i}^{2}\\
		\,&-(a_{0}B-CA_{3})f_{i}|\lambda _{i}|
		+C(B+A_{2}+A_{3})\Big(1+\sum f_{i}+F^{\tau}\Big).
	\end{aligned}
\end{equation}
Notice that
\begin{equation}\label{uii-s-12}
	\frac{A_{3}}{2}\sum_{i \neq r}f_{i}\lambda _{i}^{2} \geq 2a_{0}B
	\sum_{i \neq r}f_{i}|\lambda _{i}|-\frac{2(a_{0}B)^{2}}{A_{3}}\sum f_{i}.
\end{equation}
Thus, we derive from \eqref{uii-s-11} and \eqref{uii-s-12} that
\begin{equation}\label{uii-s-13}
	\begin{aligned}
		\mathcal{L}\varPsi \leq\,& -(A_{1}+B)\varepsilon\Big(1+\sum f_{i}+F^{\tau}\Big)
		-(a_{0}B-CA_{3})f_{i}|\lambda _{i}|\\
		\,&+C(B+A_{2}+A_{3})\Big(1+\sum f_{i}+F^{\tau}\Big)+\frac{2(a_{0}B)^{2}}{A_{3}}\sum f_{i}.
	\end{aligned}
\end{equation}

If $\lambda _{r}< 0$, similarly to \eqref{uii-s-13},
we have
\begin{equation}\label{uii-s-14}
	\begin{aligned}
		\mathcal{L}\varPsi \leq\,& -(A_{1}+B)\varepsilon\Big(1+\sum f_{i}+F^{\tau}\Big)
		-(a_{1}B-CA_{3})f_{i}|\lambda _{i}|\\
		\,&+C(B+A_{2}+A_{3})\Big(1+\sum f_{i}+F^{\tau}\Big)+\frac{2(a_{1}B)^{2}}{A_{3}}\sum f_{i},
	\end{aligned}
\end{equation}
where $a_{1}=\sup_{\mathcal{P} M_T}Ke^{K(\ul u-u)}$.

Checking \eqref{uii-s-13} and \eqref{uii-s-14}, we can choose
$A_{1}\gg A_{2}\gg A_{3}\gg 1$ and $A_{1}-B\gg a_{1} B\geq a_{0} B\gg A_{2}\gg A_{3}$ in \eqref{uii-s-13} and \eqref{uii-s-14}
such that \eqref{eq3-5} holds and $\varPsi \geq K_{1}\rho ^{2}$ in $M_{T}^{\delta}$.
\end{proof}

By \eqref{eq3-1} and \eqref{eq3-5}, we can use Lemma~\ref{barrier} to choose suitable $\delta$, $N$
and $A_{1}\gg A_{2}\gg A_{3}\gg 1$ such that in $M_{T}^{\delta}$,
$\mathcal{L}(\varPsi \pm \nabla_{\alpha}(u-\phi))\leq 0$, and $\varPsi \pm \nabla_{\alpha}(u-\phi)\geq 0$
on $\mathcal{P} M_{T}^{\delta}$. Then it follows from the maximum principle that
$\varPsi \pm \nabla_{\alpha}(u-\phi)\geq 0$ in $M_{T}^{\delta}$ and therefore
\begin{equation}
\label{pbs-E130'}
|\nabla_{n\alpha} u (x_0, t_0)| \leq \nabla_n \varPsi (x_0, t_0) \leq C,
\;\; \forall \; \alpha < n.
\end{equation}

It remains to show that
\begin{equation}
\label{unn-0}
\nabla_{nn} u (x_{0}, t_{0}) \leq C
\end{equation}
since $\triangle u -u_{t}+tr A> 0$. We shall use an idea of Trudinger \cite{Trudinger.1995}
to prove that there exist uniform positive constants $c_{0}$, $R_{0}$ such that for all $R > R_{0}$,
$(\lambda'[U], R, -u_{t}) \in \Gamma$ and
$$
f(\lambda'[U], R, -u_{t}) \geq \psi[u]+c_{0}\,\,\mbox{on}\;\overline{SM_{T}},
$$
which implies \eqref{unn-0} by Lemma 1.2 in \cite{CaffarelliNirenbergSpruck.1985}, where $\lambda'[U]=(\lambda_{1}', \ldots, \lambda_{n-1}')$
denote the eigenvalues of the $(n-1) \times (n-1)$ matrix
$\{U_{\alpha \beta} \}_{1 \leq \alpha, \beta \leq (n-1)}$ and $\psi [u] = \psi (\cdot, \cdot, \nabla u)$.
Define
\[\widetilde{F} (U_{\alpha \beta},-u_{t}) \equiv \lim_{R \rightarrow + \infty}
f (\lambda' (\{U_{\alpha \beta}\}), R, -u_{t})\]
and consider
\[m \equiv \min_{(x, t) \in \overline{S M_T}}
\Big(\widetilde{F} (U_{\alpha \beta} (x, t), -u_{t}(x, t))  - \psi [u] (x, t)\Big).
\]
Note that $\widetilde{F}$ is concave and $m$ is monotonically increasing with respect to $R$, and that
\[
c \equiv
\min_{(x, t) \in \overline{S M_T}}
\Big(\widetilde{F} (\ul U_{\alpha \beta} (x, t), - \ul u_t (x, t))  - \psi [\ul u] (x, t)\Big)>0
\]
when $R$ is sufficiently large.

We shall show $m > 0$ and
we may assume $m < c/2$ (otherwise we are done) and
suppose $m$ is achieved at a point $(x_{0}, t_{0}) \in \overline{S M_T}$.
Choose local orthonormal frames around $x_0$ as before and assume
$\nabla_{n n} u (x_0, t_0) \geq \nabla_{n n} \ul u (x_0, t_0)$.
Let $\sigma_{\alpha {\beta}} = \langle \nabla_{\alpha} e_{\beta}, e_n \rangle$
and
\[\widetilde{F}^{\alpha\beta}_{0} = \frac{\partial \widetilde{F}}{\partial r_{\alpha\beta}}
(U_{\alpha\beta} (x_{0},t_{0}), -u_{t}(x_{0},t_{0})),\]
\[\widetilde{F}^{\tau}_{0} = \frac{\partial \widetilde{F}}{\partial \tau}
(U_{\alpha\beta} (x_{0},t_{0}), -u_{t}(x_{0},t_{0})).\]
Note that
$\sigma_{\alpha \beta} = \varPi (e_\alpha, e_\beta)$ on $\partial M$
and by \eqref{hess-a200}, we have,
at $(x_0, t_0)$,
\begin{equation}
\label{pbs-c-225}
\begin{aligned}
	\nabla_n (u - \ul{u}) \widetilde{F}^{\alpha \beta}_0 \sigma_{\alpha {\beta}}
	\geq \,& \widetilde{F}(\ul{U}_{\alpha {\beta}},-\ul u_{t}) - \widetilde{F}(U_{\alpha {\beta}}, -u_{t})
	+\widetilde{F}^{\tau}_{0}(\ul u_{t}-u_{t})\\
	\,&  + \widetilde{F}^{\alpha {\beta}}_0 (A^{\alpha \beta}[u] - A^{\alpha \beta}[\ul u]) \\
	\geq \,& \frac{c}{2} + H [u] - H[\ul u]\\
	\geq \,& \frac{c}{2}+H_{p_{n}} \nabla_n (u - \ul{u}),
\end{aligned}
\end{equation}
where $H [u] = \widetilde{F}^{\alpha {\beta}}_0 A^{\alpha \beta} [u] - \psi [u]$.
The last inequality is from the regularity of $-A$ and the convexity of $\psi$ with respect to $p$.

Note that $-A$ is regular, which means $A^{\alpha\beta}$ is concave respect to $p_{n}$ and $\ul u$ is strict subsolution near the boundary, we have $H_{p_{n}{p_{n}}}\leq 0$ and
$$0<\nabla_{n}(u-\ul u)<c_{4}$$
for some positive constant $c_{4}$.
It follows from \eqref{pbs-c-225} that, at $(x_{0}, t_{0})$,
\begin{equation}\label{n1}
\kappa-H_{p_{n}}\geq \frac{c}{2c_{4}}>0,
\end{equation}
where $\kappa = \widetilde{F}^{\alpha {\beta}}_0 \sigma_{\alpha {\beta}}$.

Let $\vartheta(x,t)=\kappa(x,t)-H_{p_{n}}(x,t,\nabla'\varphi(x,t), \nabla_{n}u(x_{0},t_{0}))$.
Since $\nabla_{\alpha}u=\nabla_{\alpha}\ul u=\nabla_{\alpha}\varphi$ on $\overline{SM_{T}}$, we derive
\begin{equation}\label{n2}
\vartheta(x,t)>c_{5}\;\; \mbox{on}\; \partial M^{\delta}_{T}\cap \overline{SM_{T}}
\end{equation}
for some small positive constant $c_{5}$, where $\nabla' \varphi=(\nabla_{1}\varphi, \ldots, \nabla_{n-1}\varphi)$.

Next, since $H$ is concave with respect to $p_{n}$, we have
\begin{equation}\label{n3}
\begin{aligned}
	\,&H(x,t,\nabla'\varphi, \nabla_{n}u(x_{0},t_{0}))-H(x,t,\nabla'\varphi, \nabla_{n}u)\\
	\geq\,& H_{p_{n}}(x,t,\nabla'\varphi, \nabla_{n}u(x_{0},t_{0}))(\nabla_{n}u(x_{0},t_{0})-\nabla_{n}u)
\end{aligned}
\end{equation}
on $\overline{SM_T}$.

On the other hand, since $u_t = \ul u_t = \varphi_t$ on $\overline{SM_T}$,
by the concavity of $\widetilde{F}$, we have
\begin{equation}
\label{pbs-c-210}
\begin{aligned}
	\,& H(x,t,\nabla'\varphi, \nabla_{n}u(x,t))-H(x_{0},t_{0},\nabla'\varphi(x_{0},t_{0}), \nabla_{n}u(x_{0},t_{0}))\\
	\,&+\widetilde{F}^{\alpha {\beta}}_0(\nabla_{\alpha\beta}u-\nabla_{\alpha\beta}u(x_{0},t_{0}))
	+\widetilde{F}^{\tau}_{0} \varphi_t-\widetilde{F}^{\tau}_{0} \varphi_t (x_0, t_0)\\
	=\,&\widetilde{F}^{\alpha {\beta}}_0 U_{\alpha  {\beta}} - \psi [u] - \widetilde{F}^{\tau}_{0}u_t
	- \widetilde{F}^{\alpha {\beta}}_0 U_{\alpha  {\beta}} (x_0, t_0) + \psi [u] (x_0, t_0)
	+\widetilde{F}^{\tau}_{0} u_t (x_0, t_0)\\
	\geq \,& \widetilde{F} (U_{\alpha  {\beta}}, -u_{t}) - \psi [u] - m \geq 0
\end{aligned}
\end{equation}
on $\overline{SM_T}$.
It follows from \eqref{hess-a200}, \eqref{n3} and \eqref{pbs-c-210} that
\begin{equation}\label{n4}
\begin{aligned}
	\,&  -\vartheta(\nabla_{n}(u-\varphi)-\nabla_{n}(u-\varphi)(x_{0},t_{0}))
	\\
	\geq\,&  \widetilde{F}^{\alpha \beta}[\nabla_{n}(u-\varphi)(x_{0},t_{0})(\sigma_{\alpha\beta}(x_{0},t_{0})-\sigma_{\alpha\beta})
	+\nabla_{\alpha\beta}\varphi(x_{0},t_{0})-\nabla_{\alpha\beta}\varphi]\\
	\,&  +H(x,t,\nabla'\varphi, \nabla_{n}u(x_{0},t_{0}))-H(x_{0},t_{0},\nabla'\varphi(x_{0},t_{0}), \nabla_{n}u(x_{0},t_{0}))\\
	\,&  +H_{p_{n}}(x,t,\nabla'\varphi, \nabla_{n}u(x_{0},t_{0}))(\nabla_{n}\varphi(x_{0},t_{0})-\nabla_{n}\varphi)
	+\widetilde{F}^{\tau}_{0} \varphi_t (x_0, t_0)\\
	\,&  -\widetilde{F}^{\tau}_{0} \varphi_t\\
	:=  \,&  \Theta(x,t).
\end{aligned}
\end{equation}
From the form of the function $\Theta(x,t)$ in \eqref{n4}, since $\Theta(x_{0},t_{0})=0$, we have,
on $\partial M^{\delta}_{T}\cap \overline{SM_{T}}$,
\begin{equation}\label{n5}
\begin{aligned}
	\nabla_{n}(u-\varphi)-\nabla_{n}(u-\varphi)(\tilde{x}_{0})
	\leq \vartheta^{-1}\Theta(x,t)\\
	\leq \emph{l}(\tilde{x}-\tilde{x}_{0})+\tilde{C}(\rho^{2}+(t-t_{0})^{2}),
\end{aligned}
\end{equation}
where $\tilde{x}=(x,t)$, $\emph{l}$  is a linear function of $\tilde{x}-\tilde{x}_{0}$ with $\emph{l}(0)=0$, and the constant $C$ depends on $|u|_{C^{1}}$ and other known data.

Define
\[
\varPhi =  \nabla_{n}(u-\varphi)-\nabla_{n}(u-\varphi)(\tilde{x}_{0})-\emph{l}(\tilde{x}-\tilde{x}_{0})-\tilde{C}(t-t_{0})^{2}.
\]
By extending $\varphi$ smoothly to the interior near the boundary to be constant in the normal direction, By \eqref{eq3-1}, we have
\[
\mathcal{L}\varPhi\leq C(1+\sum f_{i}+\sum f_{i}|\lambda_{i}|+F^{\tau}).
\]

We see from \eqref{pbs-c-210} and \eqref{hess-a200} that
$\varPhi \geq 0$ on $\overline{S M_T}$ and $\varPhi (x_0, t_0) = 0$.
Therefore, by the compatibility condition \eqref{comp2}, we have, when $\delta$
is sufficiently small, $\varPsi \geq 0$ on $\mathcal{P}M_{\delta}$.

Therefore, by Lemma~\ref{barrier}, we can choose suitable $\varPsi$ such that
\begin{equation}
\label{pbs-cma-106}
\left\{ \begin{aligned}
	\mathcal{L} (\varPsi-\varPhi ) \leq  0 \;\; &\mbox{ in $M^{\delta}_{T}$},  \\
	\varPsi-\varPhi \geq 0 \;\; &\mbox{ on $\mathcal{P} M^{\delta}_{T}$}.
\end{aligned} \right.
\end{equation}
By the maximum principle we find $\varPsi\geq\varPhi$ in $M^{\delta}_{T}$.
It follows that $\nabla_n \varPhi (x_0, t_0) \leq  \nabla_n \varPsi (x_0, t_0) \leq C$.


Therefore, we have an {\em a priori}
upper bound for all eigenvalues of $\{U_{ij} (x_0, t_0)\}$ and hence its eigenvalues are contained in a compact
subset of $\Gamma$ by \eqref{f5},
and we see $m > 0$ by \eqref{f1}.

Consequently, there exist positive $c_{6}$ and $R_{0}$ such that
$$
(\lambda'(\widetilde{U}(x,t)), R, -u_{t}(x, t))\in \Gamma
$$
and
\[
f(\lambda'(\widetilde{U}(x,t)), R, -u_{t}(x, t)) \geq \psi(x, t)+c_{6}
\]
for all $R > R_{0}$ and $(x, t)\in \overline{SM_{T}}$

For $i=1,\ldots, n-1$, Lemma 1.2 in \cite{CaffarelliNirenbergSpruck.1985} means $\lambda'_{i}=\lambda_{i}+o(1)$ if $|U_{nn}|$ tends to infinity. Therefore, we have
\[
f(\lambda(U), -u_{t})>\psi
\]
for unbounded $|U_{nn}|$, which leads a contradiction and therefore \eqref{unn-0} holds.




\end{document}